%% file: main.tex
\documentclass[letterpaper, 10 pt, conference]{ieeeconf}  

\IEEEoverridecommandlockouts                              
\usepackage{cite}
\usepackage{xargs}                      
\usepackage[noend]{algorithmic}
\usepackage{mathematix}
\usepackage{hyperref}
\usepackage{nicefrac}
\usepackage{microtype}
\usepackage{units}
\usepackage{framed,lipsum}
\usepackage[capitalize]{cleveref}

\usetikzlibrary{positioning}
\tikzset{main node/.style={circle,fill=blue!20,draw,minimum size=1cm,inner sep=0pt},
            }
\usetikzlibrary{patterns}
            
\hypersetup{
    bookmarks=true,         
    unicode=false,          
    pdftoolbar=true,        
    pdfmenubar=true,        
    pdffitwindow=false,     
    pdfstartview={FitH},    
    pdftitle={Risk-averse risk-constrained optimal control},    
    pdfauthor={P. Sopasakis, M. Schuurmans and P. Patrinos},     
    pdfkeywords={},
    pdfsubject={Confidential},   
    pdfnewwindow=true,      
    colorlinks=true,        
    linkcolor=red,          
    citecolor=NavyBlue,         
    filecolor=magenta,      
    urlcolor=NavyBlue           
}

\newcommand{\cl}{\operatorname{cl}}

\newcommand{\tttop}{{\scriptscriptstyle\top}}
\newcommand{\probSimplex}{\mathcal{D}_{n}}
\newcommand{\K}{\mathcal{K}}
\newcommand{\dkl}{\operatorname{D}_{\mathrm{KL}}}

\newcommand{\crm}[1]{\rho_{\mid {}#1}}

\title{\LARGE \bf
Risk-averse risk-constrained optimal control
}

\author{Pantelis~Sopasakis$^{\dagger}$, Mathijs Schuurmans$^{\ddag}$ and Panagiotis~Patrinos$^{\ddag}$%
\thanks{$^{\dagger}$P. Sopasakis is with Queen's University Belfast, School of Electronics, Electrical Engineering and Computer Science,
               Centre for Intelligent Autonomous Manufacturing Systems, BT9 5AH, Northern Ireland, UK.
               Email: \texttt{p.sopasakis@qub.ac.uk}}%
\thanks{$^{\ddag}$M. Schuurmans and P. Patrinos are with the Department 
	of Electrical Engineering (\textsc{esat-stadius}), KU Leuven, 
	Kasteelpark Arenberg 10, 3001 Leuven, Belgium.
        {\tt\small \{mathijs.schuurmans, panos.patrinos\}@esat.kuleuven.be}.}%
\thanks{This work is accompanied by an MIT-licensed open-source toolbox
        which is available at \url{https://github.com/kul-forbes/risk-averse}.}%
\thanks{The work of the second and third authors was supported by: FWO projects: No. G086318N; No. G086518N; 
        Fonds de la Recherche Scientifique -- FNRS, the Fonds Wetenschappelijk Onderzoek 
        -- Vlaanderen under EOS Project No. 30468160 (SeLMA), 
        Research Council KU Leuven C1 project No. C14/18/068 and the 
        Ford--KU Leuven Research Alliance project No. KUL0023.}
}%

\begin{document}

\maketitle
\thispagestyle{empty}
\pagestyle{empty}

\begin{abstract}
\input{abstract.tex}
\end{abstract}



\input{intro.tex}

\input{risk_measures.tex}

\input{scenario_trees.tex}

\input{risk_optimal_control.tex}

\input{computational.tex}


\input{example.tex} 

\input{conclusion.tex}

\bibliographystyle{ieeetr}
\bibliography{main}

\appendix
\input{appendix.tex}
\end{document}

%% file: abstract.tex
Multistage risk-averse optimal control problems with nested conditional risk 
mappings are gaining popularity in various application domains. Risk-averse
formulations interpolate between the classical expectation-based stochastic
and minimax optimal control. This way, risk-averse problems aim at hedging 
against extreme low-probability events without being overly conservative.
At the same time, risk-based constraints may be employed either as surrogates
for chance (probabilistic) constraints or as a robustification of 
expectation-based constraints.
Such multistage problems, however, have been identified as particularly hard to solve.
We propose a decomposition method for such nested problems
that allows us to solve them via efficient numerical optimization methods.
Alongside, we propose a new form of risk constraints which accounts for the 
propagation of uncertainty in time.

%% file: intro.tex
\section{Introduction}
\subsection{Background, motivation and contributions}

Risk measures in stochastic optimal control serve two purposes: firstly, they
allow to account for inexact knowledge of the underlying probability 
distribution --- which in most cases is merely estimated --- and, secondly, 
offer a flexible framework which interpolates between worst-case and 
expectation-based (risk-neutral) formulations \cite{shapiro2014lectures,pflugBook}. 
Risk-averse optimal control aims at optimizing the expectation of a (random) 
cost function accounting for the worst-case probability distribution --- a general
approach which has been termed \textit{distributionally robust optimization}%
~\cite{Calafiore2006}.

In several applications it is desirable to impose constraints on random 
quantities in a probabilistic fashion (typically in the form of probabilistic 
or expectation constraints), yet in doing so one should take into account the 
\textit{ambiguity} associated with the probability distribution%
~\cite{bertsimas-brown}. Risk constraints can be interpreted as ambiguous 
expectation constraints \cite{ben-tal-soft-robust,cccNemirovski} and are often 
employed as surrogates for chance constraints~\cite{nem2012,ChowNIPS} 
in order to avoid having to resort to computationally demanding methods such as 
integer programming~\cite{shen2014}. 

Risk-averse optimal control formulations are nowadays making 
their way in applications such as 
power systems \cite{mgRisk2018}, and economics \cite{pflugBook} as their 
favorable properties are becoming evident.
Yet, their applicability is hindered by their complexity and computational cost
of associated multistage formulations. Multistage risk-averse optimal control 
problems amount to the optimization of a composition of several nonsmooth 
mappings \cite{shapiro2014lectures}. Typical numerical solution approaches, 
such as stochastic dual dynamic programming, fall short when faced with
large dimension of scenario-based problems \cite{Asamov2015,Collado2012ScenarioDO,Bruno2016979}.
When the involved risk measures are of the average value-at-risk type, we may 
obtain explicit solutions by multiparametric piecewise quadratic programming;
this is, however, limited to systems with with few states and small prediction 
horizons~\cite{patrinosECC2007}.

The contributions of this paper are twofold: (i) we present a novel framework 
for risk constraints using nested risk measures which aim at accounting for 
the \textit{propagation of ambiguity}; we call this new type of risk constraints,
\textit{multistage nested risk constraints}, (ii) we propose a reformulation of 
multistage risk-averse problems involving nested risk measures which facilitates
their numerical solution.

While much of the research attention has focused on particular risk measures, such as 
the average value-at-risk~\cite{Rock00CVaROptim,ChowNIPS} and the mean upper 
semi-deviation~\cite{Collado_RASPD}, 
it has been unclear how to extend existing results to more general risk measures.
Overall, our approach makes use of the dual conic representation of risk 
measures and allows for the use of arbitrary coherent risk measures in contrast 
to existing approaches which focus on specific risk measures.

\subsection{Notation}
Let \(\N_{[k_1,k_2]}\) denote the integers in \([k_1,k_2]\).
For \(z\in\Re^n\) let \(\plus{z}=\max\{0,z\}\), where the max is taken element-wise.
We denote the transpose of a matrix \(A\) by \(A^{\tttop}\).
The dual cone \(\K^*\) of a closed convex cone \(\K\subseteq\Re^n\) is the set 
\(\K^* = \{y\in\Re^n \mid y^{\top}x \geq 0, \forall x \in \K\}\). 
The relative interior of \(\K\) is denoted by \(\ri(\K)\).
A function \(f:\Re^n\to\Re\) is called \textit{lower semicontinuous} (lsc) if its
lower level sets, \(\{x {}\mid{} f(x)\leq \alpha\}\), are closed and it is called 
\textit{level bounded} if its lower level sets are bounded.

%% file: risk_measures.tex
\section{Measuring risk}
Let \(\Omega=\{\omega_i\}_{i=1}^{n}\) be a finite sample space equipped with the
discrete \(\sigma\)-algebra \(2^\Omega\) and a probability measure \(\prob\) with
\(\prob(\{\omega_i\})=\pi_i\). Without loss of generality, let us assume that 
\(\pi_i>0\). The pair \((\Omega, \prob)\) is called a \textit{probability space}.
A vector \(p\in\Re^{n}\) is called a \textit{probability vector} if \(p_i\geq 0\)
for all \(i\in\N_{[1,n]}\) and \(\sum_{i=1}^{n}p_i=1\). The set of all probability
vectors in \(\Re^{n}\) is called the \textit{probability simplex} and is denoted
by \(\probSimplex\). A real-valued \textit{random variable} over \((\Omega,\prob)\) 
is a mapping \(Z:\Omega\to\Re\) with \(Z(\omega_i)=Z_i\); this can be identified by 
the vector \(Z=(Z_1,\ldots, Z_n)\in\Re^n\).

Suppose that \(Z\) corresponds to a random cost. One possible way to extract
a characteristic index out of \(Z\) which quantifies its magnitude is to compute
its \textit{expectation}, which is
\[
      \E^{\pi}[{}Z{}] 
{}={} 
      \pi^{\tttop}Z.
\]
However, the expectation carries no deviation information and may fail to take 
into account extreme outcomes of the cost which might happen with low probability.
At the opposite end, the maximum of \(Z\) is defined as 
\[
      \max[{}Z{}] 
{}={} 
      \max_{i{}={}1,{}\ldots{}, n} Z_{i}.
\]
However, the maximum disregards the probability distribution and is likely to 
produce very conservative values. 

Risk measures are mappings \(\rho:\Re^{n}\to\Re\) which are used to derive a sure 
outcome which is no worse than \(Z\) itself. In other words, a risk measure extracts
a characteristic index taking into account the significance of high costs, which may 
happen with low probability. Trivially, the expectation and the maximum are risk 
measures.

A risk measure \(\rho:\Re^{n}\to\Re\) is said to be \textit{convex} if for all 
\(Z,Z'\in\Re^{n}\), \(c\in\Re\), \(\lambda \in [0,1]\) the following properties 
hold
\begin{enumerate}
 \item[A1.] \textit{Convexity}. 
	    \(\rho[\lambda Z + (1-\lambda) Z'] 
	      \leq \lambda \rho[Z] + (1-\lambda) \rho[Z']\),
 \item[A2.] \textit{Monotonicity}. 	    
	      \(\rho[Z]\leq \rho[Z']\), whenever \(Z_i\leq Z'_i\) for all 
	      \(i\in\N_{[1,n]}\),
 \item[A3.] \textit{Translation equivariance}. 
	    \(\rho[Z + c 1_n]= \rho[Z] + c\).
\end{enumerate}
A convex risk measure is called \textit{coherent} if it satisfies the additional
axiom \cite[Def. 6.4]{shapiro2014lectures}
\begin{enumerate}
  \item[A4.] \textit{Positive homogeneity}. 
	      \(\rho[\alpha Z] = \alpha \rho[Z]\) for all \(\alpha\geq 0\).
\end{enumerate}
Coherent risk measures are considered well behaving and the coherency axioms 
are heavily exploited in risk-averse optimization formulations.
Certain risk measures may satisfy a stronger monotonicity assumption \cite{javtaft17,shapich16,shapiro2014lectures}
\begin{enumerate}
 \item[A5.] \textit{Strict/strong monotonicity}. The risk measure \(\rho\) is called strictly (strongly) monotone 
 if \(\rho[Z] < \rho[Z']\) whenever \(Z \leq Z'\), \(\prob[Z < Z'] > 0\) (and \(\max Z < \max Z'\)).
\end{enumerate}
A monotone risk measure \(\rho\) can be regularized to produce a strictly monotone
risk measure by defining
\[
 \rho^\lambda[Z] =  (1-\lambda) \rho[Z] + \lambda \E^{\pi}[Z],
\]
for \(\lambda \in (0,1]\). Additionally, \(\rho^\lambda\) preserves the coherency of \(\rho\).

An important duality result is that all coherent risk measures can be written as 
\cite[Thm. 6.5]{shapiro2014lectures}
\begin{equation}\label{eq:rho_basic_representation}
      \rho[Z]
{}={} 
      \max_{\mu \in \Af(\pi)} \E^\mu[Z],  
\end{equation}
where  \(\Af(\pi) {}\subseteq{} \probSimplex\) is a closed and convex set of probability 
vectors which contains \(\pi\). 
We call \(\Af(\pi)\) the \textit{ambiguity set} of \(\rho\).
Equation \eqref{eq:rho_basic_representation} offers 
an interpretation of coherent risk measures: a coherent risk measure is the worst-case
expectation under inexact knowledge of the underlying probability vector \(\mu\).
For example, \(\max[Z] = \max_{\mu\in\probSimplex}\E^{\mu}[Z]\), that is, the maximum
operator reflects the total lack of probabilistic information.

A popular risk measure is the 
\textit{average value-at-risk} with parameter \(\alpha\in[0,1]\), which is 
defined as 
\begin{equation} 
 \label{eq:avar-definition}
      \AVAR_\alpha[Z]
{}={} 
      \begin{cases}
    \min\limits_{t\in\Re} 
        t+\nicefrac{1}{\alpha}\E^{\pi}\plus{Z-t},&\alpha\neq 0\\
    \max[Z],                    & \alpha=0
      \end{cases}
\end{equation}
The ambiguity set of \(\AVARa\) is
\begin{equation}
	\label{eq:ambiguity-set-avar}
 \Af^{\mathrm{avar}}_{\alpha}(\pi) {}={} \left\{
	\mu\in\Re^n 
	  \left| 
	    \textstyle\sum_{i=1}^{n} \mu_i =1, 
		0 \leq \mu_i \leq \frac{\pi_i}{\alpha} 
	  \right.
	\right\}.
\end{equation}
The average value-at-risk is a coherent, non-strongly monotone (except for \(\alpha=1\)), 
risk measure.
For \(\alpha=1\), \(\Af^{\mathrm{avar}}_1(\pi)=\{\pi\}\) and 
\(\AVAR_{1}[Z](\pi)=\E[Z]\). The maximal ambiguity set is attained for \(\alpha=0\),
i.e., \(\Af^{\mathrm{avar}}_0 {}={} \probSimplex\).
Therefore, \(\AVARa\) interpolates between 
the risk-neutral expectation operator (\(\AVAR_1\)) and the worst-case
maximum (\(\AVAR_0\)).
\begin{figure}[t]
\centering
 \includegraphics[width=0.95\columnwidth]{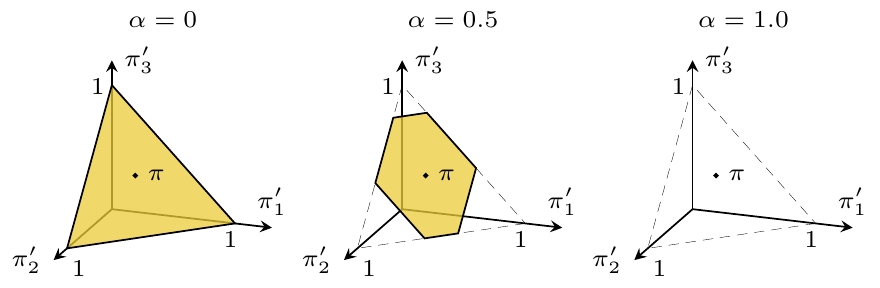}
 \caption{%
      Set \(\Af^{\mathrm{avar}}_{\alpha}(\pi)\) for different 
      values of \(\alpha\) on a space \(\Omega=
      \{\omega_1, \omega_2, \omega_3\}\) with \(\pi_1=0.2\), 
      \(\pi_2=0.3\) and \(\pi_3=0.5\). (The authors are thankful to Christian Hans, who helped
      with the Tikz code for this figure.)%
      }%
 \label{fig:avar-sets}
\end{figure}

Risk measures with polytopic ambiguity sets can be fully described by the 
set of vertices of their ambiguity sets, that is \(\Af=\conv\{\mu^{(l)}\}_{l=1}^{\kappa}\)
and \(\rho[Z] {}={} \max_{l\in\N_{[1,\kappa]}}\E^{\mu^{(l)}} [Z]\). However, the 
computation of these extreme points is a computationally demanding operation%
    \footnote{Let
              \(\Omega=\{\omega_i\}_{i=1}^{n}\) with \(\pi_i=\nicefrac{1}{n}\).
              For \(n=10\), the minimal representation of \(\mathcal{A}(\pi)\) counts \(252\) 
              vertices. For \(n=15\), the number of vertices increases to \(51480\). 
              For \(n{}={}13\), the determination of the minumum number of 
              vertices by the MPT toolbox (using Gurobi), 
              requires \(3.7\) hours.}.


Another popular coherent and strongly monotone risk measure is the 
\textit{entropic value-at-risk} at level \(\alpha\in(0,1]\), denoted by \(\EVARa\)~\cite{EVAR}, whose 
ambiguity set is given by
\begin{equation}
	\Af^{\mathrm{evar}}_{\alpha}(\pi)
{}={} \left\{ 
	  \mu \in \probSimplex 
      {}\mid{} 
		\dkl(\mu {}\|{} \pi) 
	  {}\leq{} 
		-\ln \alpha
      \right\},
\end{equation}
where 
\[
      \dkl(\mu \| \pi)
{}\dfn{} 
      \sum_{i=1}^{n} \mu_i \ln \frac{\mu_i}{\pi_i},
\] 
is the Kullback-Leibler divergence from \(\pi\) to \(\mu\).
We have \(\EVAR_{1}[Z]=\E[Z]\) and \(\lim_{\alpha \downarrow 0}\EVARa[Z] = \max[Z]\).
The entropic value-at-risk \(\EVARa\) is strongly monotone and, additionally, strictly monotone
over the space of random variables with \(\prob[\{\omega {}\mid{} Z(\omega) = \max[Z]\}] < 1-\alpha\) \cite{javtaft17}.
%

%% file: scenario_trees.tex
\section{Stochastic systems and multistage risk}

\subsection{System dynamics and scenario trees}
Consider the following discrete-time dynamical system
\begin{equation}
      \label{eq:system-dynamics}
      x_{t+1}
{}={}
      f(x_{t}, u_{t}, w_{t}),
\end{equation}
with state variable \(x_{t}\in\Re^{n_x}\), input \(u_{t}\in\Re^{n_u}\) and a disturbance
\(w_t\in\Re^{n_w}\) which is a random process. We will study the evolution of this
system throughout a finite horizon of \(N\) future time instants referred to as 
\textit{stages}. Starting from a known initial state \(x_{0}\) the system states 
evolve according to \eqref{eq:system-dynamics} as illustrated in Figure 
\ref{fig:tree}(a) giving rise to a structure known as a \textit{scenario tree}.
There have been proposed several methodologies to generate scenario trees from 
data \cite{Pflug2015scen,HeiRom09}.

The \textit{nodes} of the tree are assigned a unique index \(i\)
with \(i=0\) being the \textit{root node} which corresponds to the initial state 
\(x_0\).  The nodes at stage \(t\in\N_{[0,N]}\) are denoted by \(\nodes(t)\). 
Starting from the root node, a node \(i\) is visited with probability 
\(\pi^{i}>0\) (and \(\pi^{0}=1\)); this makes \(\nodes(t)\) a probability 
space with probability vector \(\pi_{t}=(\pi^{i})_{i\in\nodes(t)}\).

The unique \textit{ancestor} of a node \(i\in\nodes(t)\setminus \{0\}\) is denoted 
by \(\anc(i)\) and the set of \textit{children} of \(i\in\nodes(t)\) for 
\(t\in\N_{[0,N-1]}\) is \(\child(i)\subseteq\nodes(t+1)\); this becomes a probability 
space with probability vector \(\pi^{[i]}=\frac{1}{\pi^{i}}(\pi^{i_+})_{i_+\in\child(i)}\).

As shown in Fig.~\ref{fig:tree}(a), every node \(i\) of the tree is 
associated with a state value \(x^i\) and all non-leaf nodes \(i\) are assigned an input 
\(u^i\). 
Every edge connecting \(i\) with \(i_+{}\in{}\child(i)\) is associated with a disturbance \(w^{i_+}\).
The finite-horizon evolution of \eqref{eq:system-dynamics} on the 
scenario tree is described by
\begin{equation}
 \label{eq:system-dynamics-tree}
      x^{i_{+}} 
{}={}
      f(x^{i}, u^{i}, w^{i_{+}})
\end{equation}
for all \(i\in\nodes(t)\), \(t\in\N_{[0,N-1]}\) and \(i_{+}\in\child(i)\).
Note that having assigned a control action \(u^{i}\) means that decisions
are made in a causal fashion, i.e., control actions \(u_{t}\) are only allowed to 
depend on information that is available up to time \(t\).
The nodes of the tree at stage \(N\) are called \textit{leaf nodes}.

\subsection{Measuring risk on scenario trees}%
\label{sec:measuring-risk-on-scenario-trees}
\vspace{0em}
In this section we introduce the notion of conditional risk mappings which is 
essential in measuring the risk of a random cost which evolves in time 
across the nodes of a scenario tree \cite[Sec. 6.8.1]{shapiro2014lectures}.

%

\begin{figure}[!tb]
 \centering
 \includegraphics[width=0.99\linewidth]{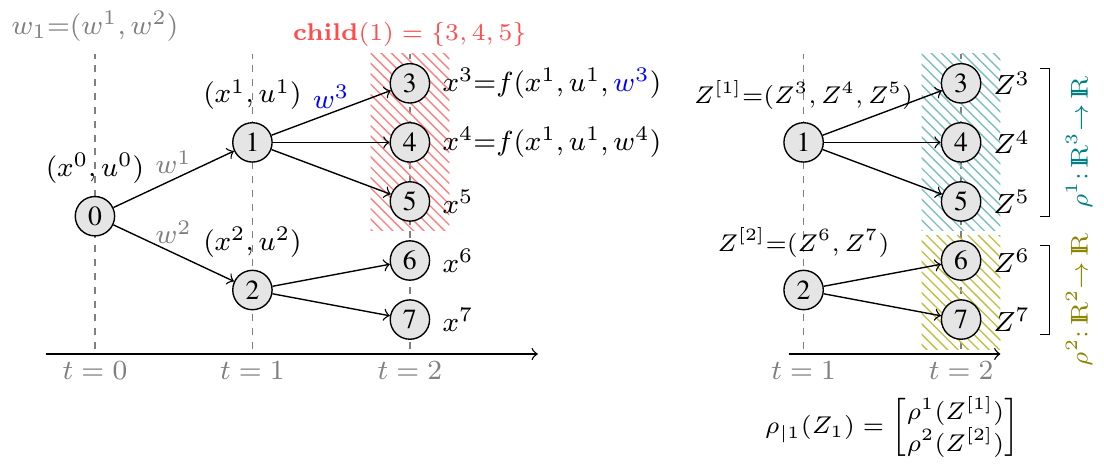}%
 \caption{(a) Structure of a scenario tree showing the evolution of the state
 of the uncertain system \eqref{eq:system-dynamics}; (b) definition of a 
 conditional risk mapping at stage \(t=1\) on a scenario tree. Note that, for example, 
 \(Z^4 = \ell_1(x^1, u^1, w^4)\) and
 \(Z_1 = (Z^3,Z^4,\ldots, Z^7)\). The conditional risk \(\crm{1}\) is a mapping 
 \(\crm{1}:\Re^{|\nodes(2)|}\to\Re^{\nodes(1)}\), 
 that is \(\crm{1}:\Re^{5}\to\Re^{2}\).}
 \label{fig:tree}
\end{figure}


For \(t\in\N_{[0,N-1]}\), let \(\ell_{t}:\Re^{n_x}\times\Re^{n_u}\times\Re^{n_w}
\to\Re\) be a stage cost function and \(\ell_N:\Re^{n_x}\to\Re\) be 
the terminal cost function. 
Such cost functions will be used in the following section to construct optimal 
control problems.

Every node \(i\in\nodes(t+1)\), \(t\in\N_{[0,N-1]}\) is associated with a 
cost value \(Z^{i} {}={} \ell_{t}(x^{\anc(i)}, u^{\anc(i)}, w^{i})\). For each
\(t\in\N_{[0,N-1]}\) we define a random variable 
\(Z_t = (Z^i)_{i\in\nodes(t+1)}\) on the probability space \(\nodes(t+1)\). 
For example, the cost at \(t=0\) is the random variable 
\(Z_0 = (Z^i)_{i\in\nodes(1)} = (\ell_0(x^0, u^0, w^{i}))_{i\in\nodes(1)}\).
At stage \(N\) the terminal cost is the random variable \(Z_N = 
(\ell_N(x^i))_{i\in\nodes(N)}\).

By defining \(Z^{[i]} {}\dfn{} (Z^{i_+})_{i_+\in\child(i)}\), \(i\in\nodes(t)\),
we partition the variable \(Z_{t} = (Z^{[i]})_{i\in\nodes(t)}\) into groups of 
nodes which share a common ancestor as shown in Fig.~\ref{fig:tree}(b).

Let \(\rho^{i}:\Re^{|\child(i)|}{\to}\Re\) be risk measures on the probability space \(\child(i)\).
For every stage \(t\in\N_{[0,N-1]}\) we may define a \textit{conditional risk mapping} 
at stage \(t\), \(\crm{t}:\Re^{|\nodes(t+1)|}\to\Re^{|\nodes(t)|}\),
as follows
\begin{equation}
\label{eq:crm-definition}
	\crm{t}[{}Z_{t}{}] 
{}={} 
	(\rho^{i}[{}Z^{[i]}{}])_{i\in\nodes(t)}.
\end{equation}
This construction is illustrated in Fig.~\ref{fig:tree}(b).

Conditional risk mappings admit a dual representation akin to that in 
Eq. \eqref{eq:rho_basic_representation}. Provided that all \(\rho^{i}\) are coherent risk
measures, \eqref{eq:crm-definition} yields 
\begin{equation}
\label{eq:crm-dual}
	\crm{t}[Z_{t}] 
{}={}
	\Big(\max_{\mu^{i}\in\Af^{i}(\pi^{[i]})} \E^{\mu^{i}} [Z^{[i]}] \Big)_{i\in\nodes(t)},
\end{equation}
where \(\Af^{i}\) is the ambiguity set of \(\rho^{i}\).
Conditional risk mappings are used to measure the risk of a multistage stochastic
process 
\(
\left(Z_{0}, \ldots, Z_{t} \right)
\) 
of random costs, which evolves on a scenario tree.  
Given a sequence 
\(
(\crm{0},\ldots,\crm{t})
\)
of conditional risk mappings, we define
\begin{equation*}
	  \varrho_{t}(Z_{0},\dots,Z_{t}) 
{}={} 
	  \crm{0} \big[ 		    
		    Z_0 {}+{} \crm{1}[{}\cdots{} 
		      {}+{} \crm{t} [{}Z_{t}{}]{}]
		  \big]
\end{equation*}
which is called a \emph{nested multistage risk} measure. 
We define the \textit{composite risk measure} at stage \(t\) as 
\begin{equation}
  \label{eq:composite_risk_measure} 
 \bar{\rho}_{t}[{}Z_{t}{}] =  \varrho_{k}(0,\dots,0,Z_{t}).
\end{equation}
If all \(\rho^{i}\) are coherent risk measures, then \(\bar{\rho}_{t}\) is a coherent 
risk measure on \(\nodes(t)\) \cite{shapiro2014lectures}.

%% file: risk_optimal_control.tex
\section{Risk-constrained risk-averse optimal control}
\subsection{Risk-averse optimal control problems}
A risk-constrained risk-averse optimal control problem with horizon \(N\) is 
defined via the following multistage nested formulation
\cite[Sec.~6.8.1]{shapiro2014lectures}%
\begin{subequations}\label{eq:risk-problem}
\begin{align}
 V^\star {}={}& \inf_{u_0}\crm{0}\bigg[\ell_0(x_0, u_0, w_0)
	+ \inf_{u_1}\crm{1}\Big[ \ell_1(x_1, u_1, w_1)\notag\\
	{}+{}&\ldots + \inf_{u_{N-1}} \crm{N-1}\big[\ell_{N-1}(x_{N-1}, u_{N-1}, w_{N-1}) \notag\\
	{}+{}& \ell_{N}(x_N)\big] \cdots
   \Big]
 \bigg]
\end{align}
subject to
\begin{align}
x_{t+1} {}={}& f(x_t, u_t, w_t),\label{eq:dynamics-constraints}\\
r_t[\phi_{j,t}(x_t, u_t, w_t)] {}\leq{}& 0, j\in\N_{[1,q_t]}, \label{eq:risk-constraints-stage}\\
r_N[\phi_{j,N}(x_N)] {}\leq{}& 0, j\in\N_{[1,q_N]}\label{eq:risk-constraints-terminal}
\end{align}
\end{subequations}
for all \(t\in \N_{[0,N-1]}\).
Constraints \eqref{eq:risk-constraints-stage} are risk constraints involving
risk measures \(r_t\) on the probability spaces \(\nodes(t+1)\) and \(r_N\) is 
a risk measure on \(\nodes(N)\). Their role is discussed in Section
\ref{sec:risk-constraints}.
The infima in \eqref{eq:risk-problem} are taken with respect to causal control functions \(u_t\).

The above nested formulation amounts to minimizing the nested
multistage cost \(\varrho_{N-1}(\ell_0(x_0, u_0, w_0),\ldots, \ell_N(x_N))\) subject to the 
system dynamics and additional constraints \cite[Sec. 6.8]{shapiro2014lectures}.
Replacing the conditional risk mappings, \(\crm{t}\), with 
conditional expectations, \(\E_{\mid{}t}\), results in
a standard expectation-based problem \cite{bertsekas-book-vol1}.
Similarly, when the underlying risks are the maximum operators, 
we obtain a minimax problem \cite{bertsekas-book-vol1}.
Therefore, risk-averse problems 
generalize risk-neutral and minimax formulations and contain them
as special cases.
Moreover, the above formulation enables the stability analysis 
of associated model predictive control formulations%
~\cite{chow-risk-averseMPC,HeSoPaBe_RAMPC}.
%
%

\subsection{Risk constraints}\label{sec:risk-constraints}
At each stage \(t\in\N_{[0,N-1]}\), let us define \(q_t\) functions 
\(\phi_{j,t}:\Re^{n_x}\times\Re^{n_u}\times\Re^{n_w}\to\Re\), \(j\in\N_{[1,q_t]}\).
At stage \(N\), we also define \(q_N\) functions \(\phi_{j,N}:\Re^{n_x}\to\Re\),
\(j\in\N_{[1,q_N]}\).
Reciting \cite{risk-quadrangle}, our objective is to impose that 
``\(\phi_{j,t}\) are adequately \(\leq{} 0\),'' for \(t\in\N_{[0,N]}\), in a 
probabilistic sense.

Let \(G_{j,t} = \phi_{j,t}(x_t, u_t, w_t)\) be a real-valued random quantity 
defined at stage \(t\in\N_{[0,N-1]}\) and \(G_{j,N} = \phi_{j,N}(x_N)\).
Similar to the definition of \(Z^i\) in Sec.~\ref{sec:measuring-risk-on-scenario-trees},
at every stage \(t\in\N_{[0,N-1]}\) and node \(i\in\nodes(t+1)\), we assign values
\(G_{j,t} = ((G_{j}^{i})_{i\in\nodes(t+1)}\) for every \(j\in\N_{[1,q_t]}\). 
Analogously, we define \(G_{j,N}\) for \(j\in\N_{[1,q_N]}\)

Risk constraints may serve several purposes:
(i) the average (and the entropic) value-at-risk can be used as a convex approximation of 
chance constraints~\cite{Rock00CVaROptim}. Chance constraints of the 
form \(\prob[G_{j,t} {}\leq{} 0] \geq 1-\delta\) can be approximated
by risk constraints of the form \(\AVAR_{\delta}[G_{j,t}] \leq 0\)
(or \(\EVAR_{\delta}[G_{j,t}] \leq 0\)) --- in particular, \(\AVAR\) 
offers a tight convex approximation to chance constraints 
\cite[Sec. 4.3.3]{ben-tal-robust-book}, 
(ii) to impose ambiguous expectation constraints,
that is, constraints of the form \(\E^{\mu}[G_{j,t}] \leq 0\) 
for all \(\mu\) in a set \(\mathcal{M}\) \cite{ben-tal-soft-robust}, and lastly, 
(iii) to accommodate ambiguity in chance constraints, i.e., 
\(\prob[G_{j,t} {}\leq{} 0] \geq 1-\delta\) for all \(\prob\in\mathcal{M}\)%
~\cite{shapiro2014lectures}.
Here, we study two different risk constraint formulations on scenario
trees, namely, 
(i) stage-wise risk constraints, 
(ii) multistage nested risk constraints.

\subsubsection{Stage-wise risk constraints}

Stage-wise constraints are imposed at every stage \(t\in\N_{[0, N-1]}\)
as follows
\begin{equation}\label{eq:stage-wise-risk-constraints}
	r_{t} [{}G_{j,t}{}]
{}\leq{}
	0,
\end{equation}
for \(j\in\N_{[1,q_t]}\), where \(r_{t}:\Re^{|\nodes(t+1)|}\to\Re\) are 
risk measures and \(G_{j,t}\in\Re^{|\nodes(t+1)|}\).
At \(t=N\), similarly, we impose \(r_{N-1}[G_{j,N}]\leq 0\) for \(j\in\N_{[1,q_N]}\).
But, such risk-based constraints do not account for how the probability
distribution at stage \(t\) is generated in time; indeed, the dependence on
previous stages in \eqref{eq:stage-wise-risk-constraints} is disregarded. This 
can lead to certain pathological cases as we demonstrate in the following section.

\subsubsection{Multistage nested risk constraints}\label{sec:multistage_risk_constraints}
\begin{figure}%
 \centering%
 \hspace*{1.8em}
 \includegraphics[width=0.95\linewidth]{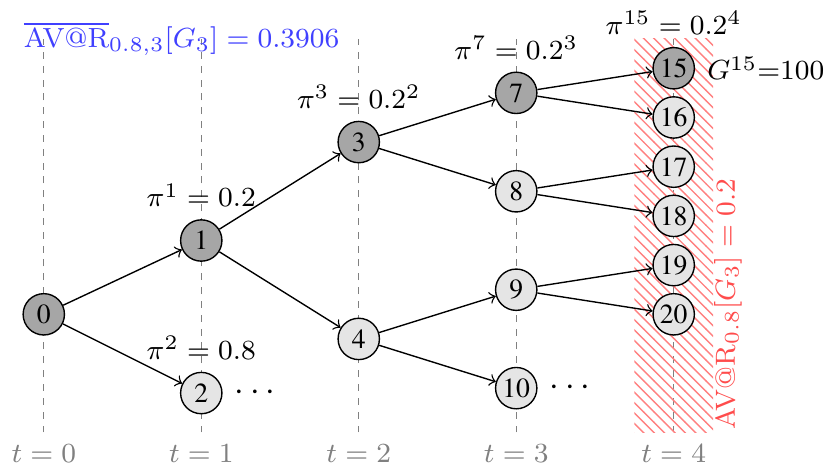}
 \caption{Motivation for the use of multistage nested risk constraints. As explained in 
 Section~\ref{sec:multistage_risk_constraints}, \(G_3=(G^{15}, G^{16}, \ldots) = (100, 0, \ldots, 0)\).}%
 \label{fig:motivation-nested}%
\end{figure}%
Consider a scenario tree generated by an iid process \((w_t)_{t\in\N_{[0,N-1]}}\) 
with \(w_t\in\{0,1\}\) as the one in Fig. \ref{fig:motivation-nested} with 
\(\prob[w_t = 0] = 0.2\) and \(\prob[w_t=1]=0.8\). 
Let functions \((\phi_t)_{t\in\N_{[0,N]}}\) be such 
that \(G^i=0\) for \(i\neq 15\) and \(G^{15}=10^2\).
The (nominal) probability of node \(i=15\) is 
\(\pi^{15} = 0.2^4 = 0.0016\). Suppose that the probability has 
been misestimated and the actual one is \(\prob'[w_t = 0] = 0.25\) and 
\(\prob'[w_t = 1] = 0.75\). This probability vector is within 
\(\Af^{\mathrm{avar}}_{0.8}(\prob)\). On the other hand, we have that \(\pi^{\prime{}15}
=0.25^4 = 0.0039\), but the ambiguity set \(\Af^{\mathrm{avar}}_{0.8}\) on \(\nodes(4)\)
contains no \(\mu\) such that \(\mu[\{15\}]=0.0039\).
In other words, the stage-wise risk fails to describe how 
ambiguity may build up and propagate in time.
This motivates the introduction of \textit{multistage nested risk constraints}
of the form
\begin{equation}\label{eq:nested-risk-constraint}
	\bar{r}_t[{}G_{j,t}{}]
{}={}	
	r_{\mid {}0}
	  \big[ r_{\mid {}1} 
	    \big[
	      \cdots
		{}r_{\mid{}t}[{}G_{j,t}{}] 
	    \big] 
	  \big]
{}\leq{}
	0
\end{equation}
Here, using the average value-at-risk with parameter \(\alpha=0.8\), 
we compute the risk \(\AVAR_{0.8}[G_3]\,{=}\,0.2\); the corresponding nested risk is 
\(\overline{\AVAR}_{0.8,3}[G_3] \,{=}\, 0.3906\).
Note that nested risk constraints neither imply nor are implied by stage-wise ones.

%% file: computational.tex
\section{Tractable reformulations}

\subsection{Conic representation of risk measures}
\label{sec:conic-risk-measures}
The ambiguity set of a coherent risk measure can be written using conic inequalities, i.e., there exist matrices $E, F$ and a vector $b$, such that   
\begin{equation}
\label{eq:conic-risk}
      \rho[Z]
{}={}
      \max_{\mu\in\Re^{n}, \nu\in\Re^{r}}
      \{
	\mu^{\tttop}Z
	{}\mid{}
	E\mu + F\nu \preccurlyeq_{\K} b
      \},
\end{equation}
where \(\K\) is a closed, convex cone and \(\nu\) is an auxiliary variable. All widely 
used coherent risk measures can be written in this form.
Tacitly, we have assumed that all admissible \(\mu\) in \eqref{eq:conic-risk} are
probability vectors and the ambiguity set of \(\rho\) is the following subset of \(\probSimplex\)
\[
      \mathcal{A}
{}={} 
      \{
	      \mu\in\Re^n 
	  {}\mid{} 
	      \exists \nu \in \Re^{r}
	  {}:{}
	      E \mu + F \nu \preccurlyeq_{\K} b 
      \}.
\]
For example, \(\AVARa\) is written as in \eqref{eq:conic-risk} with
\(r=0\) and \(E = \smallmat{I & -I & 1_n}^{\tttop}\),
\(b=\smallmat{\nicefrac{1}{\alpha}\pi^{\tttop}&0_{}&1}^{\tttop}\) and 
\(\K=\Re_{\geq 0}^{2n}\times \{0\}\). 
\(\EVARa\) can also be written in the above form. 
Let 
\(
\K^{\mathrm{e}} = 
 \cl\{(x,y,z)\in\Re^3 {}\mid{} y e^{x/y} \leq z, y>0\}
\)
be the exponential cone.
By virtue of the equivalence 
\(
x\ln(x/y)\leq t \Leftrightarrow (-t,x,y)\in \K^{\mathrm{e}}
\),
the ambiguity set \(\Af^{\mathrm{evar}}_{\alpha}(\pi)\) is
\begin{align*}
	\Af^{\mathrm{evar}}_{\alpha}(\pi)
{}={}
	\left\{ 
	  \mu\in\probSimplex 
	  \left| 
	  \hspace{-0.35em}
	  {
	  \begin{array}{l}
		    \exists \nu\in\Re^{n}: \sum_{i=1}^{n} \nu_{i} 
	      {}\leq{} 
		    -\ln \alpha, 
	      \\
		    (-\nu_{i}, \mu_{i}, \pi_{i})
	      {}\in{} 
		    \K^{\mathrm{exp}}, i\in\N_{[1,n]}
	  \end{array}
	  }
	  \right.
	  \hspace{-0.5em}
	\right\}{\!,}
\end{align*}
Lastly, note that if \(\rho\) is admits a conic representation, so does the regularized 
risk measure \(\rho^\lambda\), which has the ambiguity set 
\begin{equation*}
    \mathcal{A}^\lambda
{}={}
    \left\{ 
      \mu\in\probSimplex
      \left|
      {
      \begin{array}{l}
      \exists \nu \in \Re^{r+n} \text{ such that:}\\
      \smallmat{0\\I}
      \mu
      +
      \smallmat{E & F\\ (\lambda-1)I & 0}\nu 
      \preccurlyeq_{\K\times\{0_{n}\}}
      \smallmat{b\\\lambda \pi}
      \end{array}}
      \right.
    \right\}.
\end{equation*}

Provided that strong duality holds --- which is the case if there exist \(\mu^*\) and \(\nu^*\) 
so that \(b - E\mu^* + F\nu^* \in \ri(\K)\) \cite[Thm.~1.4.2]{BenTal2001Book} 
--- the risk measure in \eqref{eq:conic-risk} can be 
written as
\begin{equation}\label{eq:conic-risk-dual}
 \rho[Z] = \min_{y}\,
      \{
	y^{\tttop}b 
      {}\mid{}
	E^{\tttop}y = Z, F^{\tttop}y = 0, y \succcurlyeq_{\K^*} 0
      \}.
\end{equation}
We shall use this representation of risk measures to derive computationally tractable
reformulations of optimal control problems involving risks.

\subsection{Decomposition of nested formulation}

In this section we propose a computationally tractable reformulation of problem%
~\eqref{eq:risk-problem} using the following result
\begin{thm}[Risk-infimum interchangeability]\label{thm:interchangeability}
Let \(\rho : \Re^n \rightarrow \Re\) be a convex risk measure
and \(g:\Re^m \ni x \mapsto (g_1(x), \ldots, g_n(x)) \in \Re^n\) where 
\(g_i:\Re^m \rightarrow \Re\) is an lsc, 
level-bounded function over a closed set \(\emptyset \neq X \subseteq \Re^m\). 
Let \( \inf_{x\in X} g(x) {}\dfn{} \left(\inf_{x\in X} g_1(x), \dots, 
\inf_{x\in X} g_n(x)\right)\). Then 
\begin{subequations}
\begin{align} 
  \rho 	\Big[
	  \inf_{x\in X} g(x) 
	\Big] 
{}={} 
  \inf_{x\in X} 
      \rho 
	  [{}g(x){}]
  \label{eq:interchangeability_infima} 
\\
  \argmin_{x\in X} g(x) 
{}\subseteq{} 
  \argmin_{x\in X} 
      \rho 
	  [{}g(x){}]. 
  \label{eq:argmin_subset}
\end{align}
\end{subequations}
Furthermore, if \(\rho\) is \emph{strictly monotone} or \(\rho\circ g:\Re^m\to\Re\) 
is strictly convex over \(X\), then
\begin{align}
	\argmin_{x\in X} g(x) 
{}={}
	\argmin_{x\in X} \rho [{}g(x){}]. 
\label{eq:argmin_equality} 
\end{align}
\begin{proof}
The proof is given in the appendix.
\end{proof}
\end{thm}

The \textit{epigraph} of a risk measure \(\rho:\Re^n\to\Re\) is the set 
\(
	\epi \rho 
{}={}
	\{
	    (Y,\gamma)\in\Re^{n+1} 
	{}\mid{}
	    \rho[{}Y{}] {}\leq{} \gamma
	\}
\). 
When \(\rho\) is a coherent risk measure given by \eqref{eq:conic-risk-dual}, 
its epigraph is the set 
\begin{equation}
 \epi \rho = 	\left\{
			(Y,\gamma)\in\Re^{n{+}1} 
	      {}\left|{}
			\begin{array}{l}
			 \exists y \succcurlyeq_{\K^*} 0, E^{\tttop}y{=}Y,
			 \\
			 F^{\tttop}y=0, y^{\tttop}b \leq \gamma
			\end{array}			
	      {}\right.{}
		\right\}.
\end{equation}
Then, for example, stage-wise risk constraints \eqref{eq:stage-wise-risk-constraints}
are equivalent to \(r_t(\inf_{G_{j,t} \leq \eta_{j,t+1}}\eta_{j,t+1}) \leq 0\) for a 
random variable \(\eta_{t+1}\in\Re^{|\nodes(t+1)|}\).
Using Thm. \ref{thm:interchangeability}, we have that the risk constraints \eqref{eq:stage-wise-risk-constraints} 
are equivalent to the existence of \(\eta_{j,t+1}\) such that 
\(G_{j,t} \leq \eta_{j,t+1}\) and \((\eta_{t+1}, 0) \in \epi r_t\).

We shall now derive the epigraph of nested risk measures. To that end, we first
define the epigraph of a conditional risk mapping
\(
\epi \crm{t} 
{}={}
	\{
	    (Y_{t+1},Y_{t})\in\Re^{|\nodes(t+1)|+|\nodes(t)|} 
	{}\mid{}
	    \crm{t}[{}Y_{t+1}{}] 
	    {}\leq{} 
	    Y_{t}
	\}
\)
which is the Cartesian product of the epigraphs of its constituent risk measures
\begin{equation*}
	\epi \crm{t} 
 {}={} 
	\textstyle\prod_{i\in\nodes(t)} \epi\rho^i
\end{equation*}

\begin{proposition}[Nested risk epigraph]
\label{prop:epi-nested}
Let \((\crm{0},\ldots,\crm{t})\) be a sequence of coherent conditional risk 
mappings. Let \(\bar{\rho}_t\) be the corresponding nested risk measure. Its 
epigraph is
\begin{equation*}
    \epi \bar{\rho}_t 
{\,}{=}{\!}
    \left\{\hspace{-0.4em}
      \begin{array}{l}
      (Y_{t+1}, Y_0) \in \Re^{|\nodes(t+1)|+1}
	  {}\mid{} \exists (Y_j)_{j\in\N_{[1,t]}},
      \\
      Y_j\in\Re^{\nodes(j)},
      (Y_{j+1}, Y_{j}){\in}\epi \crm{j},\, j\in\N_{[0,t]}
      \end{array}
      \hspace{-0.3em}
    \right\}
\end{equation*}
\begin{proof}
The proof is given in the appendix.
\end{proof}
\end{proposition}
Using Prop. \ref{prop:epi-nested}, we may write 
\eqref{eq:nested-risk-constraint} in the form \((G_{j,t}, 0)\in\epi \bar{r}_{t}\).
Risk constraints, both stage-wise and nested, can be cast as conic constraints. 
%
%
%
By virtue of the interchangeability property in Theorem \ref{thm:interchangeability},
problem \eqref{eq:risk-problem} is written as 
\begin{equation*}
 \minimize_{u_0,u_1,\ldots, u_{N-1}}{\,}
 \crm{0}[Z_0 {\,}{+}{\,} \crm{1}[Z_1 {\,}{+}{\,} {\ldots} {\,}{+}{\,} \crm{N-1}[Z_{N-1}{\,}{+}{\,}Z_N]],
\end{equation*}
subject to \eqref{eq:dynamics-constraints}--\eqref{eq:risk-constraints-terminal}.
Similarly, this is equivalent to 
\begin{equation*}
 \minimize_{
    \substack{
	u_0,\ldots, u_{N-1},\,Z_N\leq s_N
	\\
	Z_t\leq \tau_{t+1},\, t\in\N_{[0,N-1]}
      }
    }
 {\,}
 \crm{0}[\tau_1 {\,}{+}{\,} \crm{1}[\tau_2 {\,}{+}{\,} {\ldots} {\,}{+}{\,} \crm{N-1}[\tau_{N}{\,}{+}{\,}s_N]],
\end{equation*}
subject to \eqref{eq:dynamics-constraints}--\eqref{eq:risk-constraints-terminal}, where 
\(\tau_t\in\Re^{|\nodes(t)|}\) and \(s_N\in\Re^{|\nodes(N)|}\).
Starting by epigraphically relaxing the innermost term, \(\crm{N-1}[\tau_N+s_N]
=\inf\{s_{N-1}\mid (\tau_N+s_N, s_{N-1})\in\epi\crm{N-1}\}\), proceeding backwards,
employing Thm. \ref{thm:interchangeability} and using the dual conic representation 
of risk measures, we obtain the following formulation
\begin{subequations}\label{eq:problem-simplified}
\begin{align}
  &\minimize_{
	      \substack{
		u_0,\ldots,u_{N-1}
		\\
		s_0,\ldots,s_{N},
		\tau_0,\ldots,\tau_{N}
	      }
	    } {} 
	  s_{0}
\\
&  \stt{}\,{} 
	x_0 {}={} x
	\text{ and }
	x_{t+1} 
      {}={} 
	f(x_t, u_t, w_t), 	
\\   
& \phantom{\stt{}\,{}}
	(\tau_{t+1}+s_{t+1}, s_{t})\in\epi \crm{t},
	\label{eq:constraints-epi}
\\
&\phantom{\stt{}\,{}}
	Z_N \leq s_N, Z_t \leq \tau_{t+1}, \text{ for } t\in\N_{[0,N-1]},
	\label{eq:constraints-Z-tau}
\end{align}
subject to additional risk constraints in the form \eqref{eq:risk-constraints-stage}
and \eqref{eq:risk-constraints-terminal}.
\end{subequations}

In particular, if each \(\rho^i\) is a conic risk measure which is described by 
the tuple \(E^i, F^i, b^i, \K^i\) then the above optimization problem boils 
down to
\begin{subequations}\label{eq:problem-breakdown}\label{eq:problem-simplified}
\begin{align}
  &\minimize{}
	  s^{0}
\\
&  \stt{}\,{} 
	x^0 {}={} x
	\text{ and }
	x^{i_+}
      {}={} 
	f(x^i, u^i, w^{i_+}), 	
\\   
& \phantom{\stt{}\,{}}
	y^i \succcurlyeq_{(\K^i)^*} 0, (E^i)^{\tttop}y^i = \tau^{[i]} + s^{[i]}
	\label{eq:problem-breakdown:Ey}
\\
& \phantom{\stt{}\,{}}
	(F^i)^{\tttop}y^i = 0, (y^i)^{\tttop}b^i \leq s^i
\\
& \phantom{\stt{}\,{}}
	\ell_t(x^{i}, u^{i}, w^{i_+}) \leq \tau^{i_+}, 
	\ell_N(x^{i'}) \leq s^{i'},
\end{align}
for \(i' \in \nodes(N)\), \(i\in\nodes(t)\), \(t\in\N_{[0,N-1]}\) 
and \(i_+\in\child(i)\). In \eqref{eq:problem-breakdown:Ey} we 
denote \(\tau^{[i]} = (\tau^{i_+})_{i_+\in\child(i)}\).

Suppose that the  problem is subject to stage-wise risk constraints of the form 
\eqref{eq:stage-wise-risk-constraints} at stage \(t\) with a conic risk measure 
\(r_t\) described by the tuple 
\(
	(\bar{E}_t, \bar{F}_t, \bar{b}_t, \bar{\K}_t)
\). For notational convenience, we drop the index \(j\).
\begin{align}
 \bar{y}_{t} \succcurlyeq_{\bar{\K}_t^*} 0, 
 \ {}
 \bar{E}_t^{\tttop} \bar{y}_{t} = \eta_{t},
 \ {}
 \bar{F}_t^{\tttop}\bar{y}_{t} = 0,
 \\
 \bar{y}_{t}^{\tttop} \bar{b}_t \leq 0,
 \ {}
 \phi_{t}(x^i, u^i, w^{i_+}) \leq \eta_{t}^{i_+}.
\end{align}
for \(i\in\nodes(t)\), \(i_+\in\child(i)\).
We have here introduced the additional variables \(\eta_{t}\in
\Re^{|\nodes(t+1)|}\) and \(\bar{y}_{t}\).

Similarly, suppose that the problem is subject to multistage nested risk 
constraints at stage \(t\) of the form \eqref{eq:nested-risk-constraint} where
the multistage risk is given by conic risk measures \(r^i\) described by the 
tuples \((\tilde{E}^i, \tilde{F}^i, \tilde{b}^i, \tilde{\K}^i)\). Then, 
\eqref{eq:nested-risk-constraint} leads to the following constraints
\begin{align}
	\tilde{y}_{t}^i  {}\succcurlyeq_{(\tilde{\K}^i)^*}{} 0,  
    \ {}
	(\tilde{E}^i)^\tttop \tilde{y}_{t}^i {}={} \xi_{t}^{[i]}, 
    \ {}
	(\tilde{F}^i)^\tttop \tilde{y}_{t}^i {}={} 0,
\\
	\phi_{t}(x^{i'}, u^{i'}, w^{i'_+}) \leq \xi_{t}^{i'_+},
    \ {}
	(\tilde{b}^i)^\tttop \tilde{y}_{t}^i {}\leq{} \xi_{t}^i, 
    \ {}
	\xi_{t}^0 {}={} 0,
\end{align}
\end{subequations}
for \(i'{\in}\nodes(t)\), \(i'_+{\in} \child(i')\), \(i{\in} \nodes(t')\), \(t'{\in}\N_{[1,t]}\).

In all cases, the number of decision variables and constraints increases 
linearly with the total number of nodes. Although nested risk constraints 
are translated to more constraints than their stage-wise counterparts, 
the associated complexity is of the same order of magnitude (see 
Section \ref{sec:example} for computation times).

When the system dynamics is linear (or affine) and functions \(\ell_t\) and 
\(\phi_t\) are convex in \(x\) and \(u\), then 
\eqref{eq:problem-simplified} is a convex conic problem which 
can be solved very efficiently with solvers such as MOSEK \cite{mosek},
SuperSCS \cite{superscs} and more.

Problems~\eqref{eq:risk-problem} and~\eqref{eq:problem-simplified} are equivalent 
in the sense that the optimal values of the objective function at the solution 
are the same. If all involved risk measures are strictly monotone, then the 
respective sets of minimizers are equal.

An important property that allows to establish a link between \eqref{eq:risk-problem} 
and~\eqref{eq:problem-simplified} is that of \textit{time consistency} of a policy
\((u_0,\ldots, u_{N-1})\); a policy is called time consistent if for every 
\(t=1,\ldots, N-1\), the tail \((u_t,\ldots, u_{N-1})\) is optimal conditional on
\((x_0, \ldots, x_{t})\) \cite{shapiro2014lectures}. Clearly, all solutions of 
\eqref{eq:risk-problem} are time consistent. According to \cite[Thm.~2]{shapich16},
all time consistent solutions of \eqref{eq:problem-simplified} are optimal for 
\eqref{eq:risk-problem}.

In control applications on Markovian switching systems, such as \cite{HeSoPaBe_RAMPC}, 
problem formulations akin to \eqref{eq:problem-simplified} are employed in a receding
horizon fashion: a multistage risk-averse problem is solved at each time instant
and the first control action is applied to the dynamical system.
The fact that not all policies are time consistent does not compromise the stability 
properties of the closed loop.

%% file: example.tex
\section{Illustrative example}\label{sec:example}

Suppose \((w_t)_t\) is governed by a stopped Markov process, that is 
\(w_t {}={} W_{\min(t, t_0)}\), where \((W_t)_t\) is a Markov process 
with \(m=4\) modes.  Suppose the system evolves as a Markov jump linear 
system driven by \((w_t)_t\), that is, \(f(x,u,w) = A_w x + B_w u\),
the stage cost is given by \(\ell_t(x,u,w) = x^{\tttop}Q_{w}x + u^{\tttop}R_w u\),
the prediction horizon is \(N=7\). The system dimensions are \(n_x = 2\) and \(n_u = 1\).
Matrices \(A_w\), \(B_w\), \(Q_w\) and \(R_w\) were selected randomly.
The input constraints \(-10 \leq u_t \leq 10\) are 
imposed on the control actions and suppose, for now, that no risk constraints
are imposed.

%

Consider the cumulative probability distribution of the total cost
\(\sum_{t=0}^{N-1}\ell_t(x_t, u_t, w_t) + \ell_N(x_N)\) 
in~\cref{fig:cost_cumulative}. The worst-case cost
is minimal for \(\alpha=0\) at the expense of a higher cost when moving away from 
the extremes. By contrast, for \(\alpha = 1\), the expected cost is minimal, 
yet high costs may occur with low probability. Intermediate values of \( \alpha \)
result in a trade-off between the two, effectively determining the extent
to which the right tail of the distribution of the cost is compressed. 

\begin{figure}[!htb]
 \centering
 \includegraphics[width=0.9\linewidth]{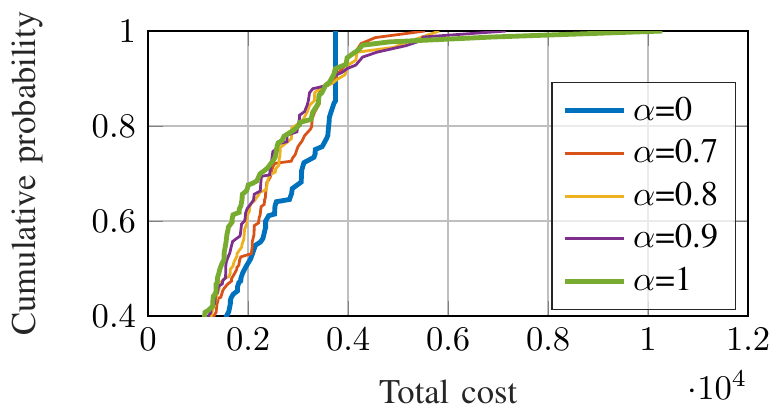}
 \caption{Cumulative probability distribution of the total cost after solving the
 risk-averse optimal control problem for different values of $\alpha$ in \( \AVARa\)
 in absense of risk constraints.}
 \label{fig:cost_cumulative}
\end{figure}

Next, consider the stage constraints function \(\phi_t(x,u,w) = \|x\|^2 - c\) with \(c=0.5\).
We use \(\AVAR_{0.5}\) in the cost and impose the stage-wise risk constraint 
\(\AVARa(\phi_t(x_t, u_t, w_t)) \leq 0\) at all stages \(t\in\N_{[N-4, N-1]}\). 
For the robust case ($\alpha = 0$), there was no feasible solution.
As $\alpha$ increases, \cref{fig:constraints_avar} shows that constraint 
violations occur in larger fractions of the realisations. Also note that since 
$\AVARa$ bounds the $(1-\alpha)$-quantile function, it is guaranteed 
that \(\prob[\|x_t\| \leq c] \geq 1-\alpha\).

\begin{figure}[!htb]
 \centering
 \includegraphics[width=0.9\linewidth]{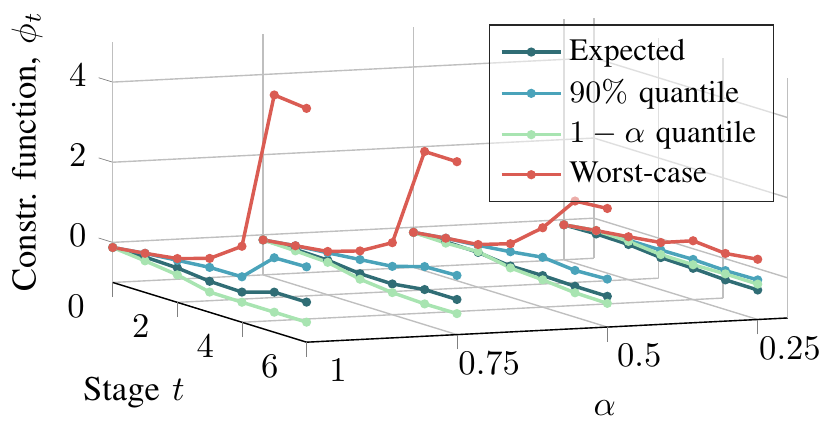}
 \caption{The constraint function \(\phi_t\) over time for a risk-averse optimal control problem
 with different values of \(\alpha\) in the average value-at-risk.}
 \label{fig:constraints_avar}
\end{figure}

\cref{fig:timings-evar-vs-avar} shows the complexity of the optimization
problem with respect to the number of scenarios. We fixed \(N=12\) and controlled 
the number of scenarios and nodes with the \textit{branching horizon} \(t_0\).
Although nested risk constraints increase the 
problem size, the asymptotic complexity remains linear in the number of nodes.
Moreover, for both constraint types, problems using \(\AVAR\) with up to 200 scenarios
can still be solved in well under a second.
Since MOSEK V. 8~\cite{mosek} does not support exponential constraints, it cannot be   
used to solve problems involving \(\EVAR\). For that reason, we resort to SuperSCS~\cite{superscs}. As shown in 
\cref{fig:timings-evar-vs-avar}, \(\EVAR\)-based problems are solved at a significantly higher runtime. 

\begin{figure}[!htb]    
 \centering 
 \begin{minipage}{0.49\linewidth}   
 \centering 
 \includegraphics[width=\linewidth]{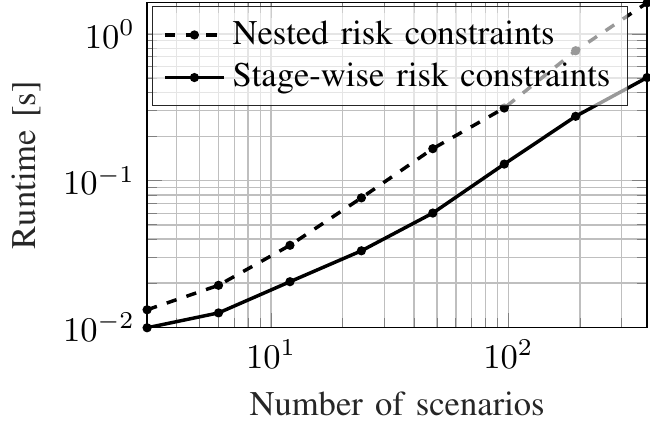}
 \end{minipage} \hfill  
 \begin{minipage}{0.49\linewidth}   
 \centering 
 \includegraphics[width=\linewidth]{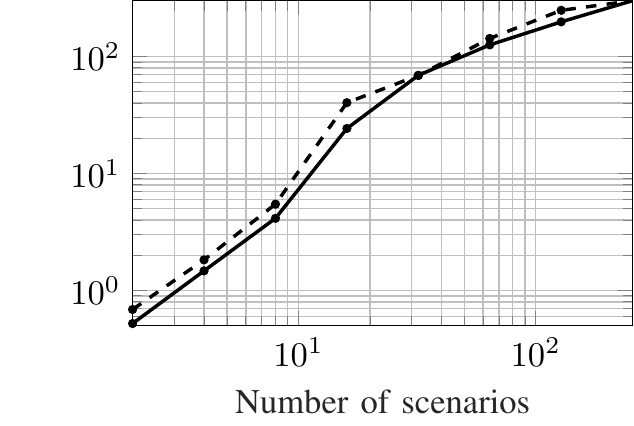}  
 \end{minipage}
\caption{Solver time versus the number of scenarios in the tree 
(computed on an Intel Core i7-7700K CPU at 4.20GHz). \textit{(left)} \(\AVAR\) (MOSEK~\cite{mosek}); \textit{(right)} \(\EVAR\) (SuperSCS~\cite{superscs}).} 
\label{fig:timings-evar-vs-avar}    
\end{figure}

%% file: conclusion.tex
\section{Conclusions}
We presented a decomposition methodology for nested 
conditional risk mappings in multistage risk-averse optimal control problems. 
In the common case where the system dynamics is linear and the state and input 
constraints are convex, the original multistage nested 
problem is cast as a conic problem, which can be solved very 
efficiently and is suitable for real-time embedded applications, provided that 
the tree contains a moderate number of nodes. 
The proposed approach hinges on the convex dual formulation of conic risk 
measures.
Future work will focus on tailored numerical methods to exploit the
tree structure to solve such problems fast, efficiently and using 
parallelization (e.g., on GPUs \cite{SamSopBemPat17b,SamSopBemPat15}).

%% file: appendix.tex
\begin{proof}[Proof of Theorem~\ref{thm:interchangeability}]
{
Let \(\rho : \Re^n \rightarrow \Re\) be a convex risk measure and 
\(g {}:{} \Re^m \ni x \mapsto  \left( g_1(x), g_2(x), \dots, g_n(x) \right) \in
\Re^n\). 
Define 
\(
  g^{\star} = \inf_{x \in X} g(x)
\); we know that \(\inf_{x \in X} g_i(x)\) are finite because of \cite[Thm. 1.9]{rockafellar2011variational}.
For \(\epsilon > 0\), define 
\[
 \B_{\epsilon}^g = \left\{x \in X \mid g_i(x) \le g_i^{\star} + \epsilon, i = 1,\ldots,n \right\}.
\]
By the definition of infimum, \(\B_{\epsilon}^g\) are nonempty and nested 
(\(\B_{\epsilon'}^g \subseteq \B_{\epsilon}^g\) for \(\epsilon' \le \epsilon\)).
For \(x \in \B_{\epsilon}^g\) we have \(g^{\star} \le g(x) \le g^{\star} + \epsilon\).	
Using the monotonicity property of \(\rho\) (A2) we obtain	
\begin{align}
\label{eq:risk_measures_sandwich}
	\rho[{}g^{\star}{}] 
{}\leq{} 
	\rho[{}g(x){}],
\end{align}
for all \(x \in \B_{\epsilon}^g\).
By taking the infimum on both sides of~\eqref{eq:risk_measures_sandwich} 
we obtain 
\begin{subequations}
 \begin{align}
	\rho[{}g^{\star}{}] 
{}\leq{} 
	\inf_{\epsilon > 0}\inf_{x \in \B_{\epsilon}^g} \rho[{}g(x){}] 
{}={} 
	\inf_{x \in X} \rho[{}g(x){}]. \label{eq:rho_fstar_under}
\end{align}
Conversely, take \(x^{\epsilon}\in\B_{\epsilon}^g\). As \(\epsilon{}\downarrow{}0\),
\(g(x^\epsilon)\to g^\star\) and because \(\rho\) is continuous, 
\(
	\rho[{}g(x^\epsilon){}]
{}\to{}
	\rho[{}g^\star{}]
\).
Since \(\inf_{x\in X} \rho[{}g(x){}] \leq \rho[{}g(x){}]\), 
\begin{equation}
\label{eq:rho_fstar_over}
	\inf_{x\in X} \rho[{}g(x){}]
{}\leq{} 
	\rho[{}g^\star{}]. 
\end{equation}
\end{subequations}
By~\eqref{eq:rho_fstar_under} and \eqref{eq:rho_fstar_over} we have 
established~\eqref{eq:interchangeability_infima}.

Let us assume now that \(\argmin_{x \in X} g(x)\) is a nonempty set. 
For any \(x^{\star} \in \argmin_{x \in X} g(x)\) it holds by definition that
\(g(x^{\star}) = \inf_{x \in X} g(x)\).
Then, by the property established above it holds that 
\(
	\rho[{}g(x^{\star}){}] 
{}={} 
	\inf_{x \in X} \rho[{}g(x){}]
\), therefore, \eqref{eq:argmin_subset} holds true. 

If \(\rho\circ g\) is strictly convex, then
the minimizer is unique, therefore \eqref{eq:argmin_equality}  holds.
Assume that risk measure \(\rho\) is strictly monotone (see Condition A5)
and there exists  \(\bar{x} \in \argmin_{x \in X} \rho \left( g(x) \right)\),
but \(\bar{x} \notin \argmin_{x \in X} g(x)\). Then, \(g(x^\star) < g(\bar{x})\) 
which, by strict monotonicity, implies \(\rho[ {}g(x^\star){} ] < \rho[{}g(\bar{x}){}]\) 
leading to contradiction.
}
\end{proof}
\vspace{1.5em}
\begin{proof}[Proof of Proposition~\ref{prop:epi-nested}]
{
For \(Y_{t+1}\in\Re^{|\nodes(t+1)|}\) and \(Y_0\in\Re\),
we have that \((Y_{t+1}, Y_0)\in\epi \bar{\rho}_t\) if \(\bar{\rho}_t[Y_{t+1}]\leq Y_0\) and
using Theorem~\ref{thm:interchangeability}
\begin{align*}
    \bar{\rho}_t[Y_{t+1}] 
=& 
    \crm{0}[{}\cdots{}\crm{t-1}[\crm{t}[Y_{t+1}]]]
\\
=&
    \crm{0}\Big[{}\cdots{}\crm{t-1}\big[\inf_{\crm{t}[Y_{t+1}]\leq Y_{t}}Y_{t}\big]\Big]
\\
=&
    \inf_{(Y_{t+1}, Y_{t})\in \epi \crm{t}}\crm{0}\big[{}\cdots{}\crm{t-1}[Y_{t}]\big]
\\
=&
    \inf_{(Y_{t+1}, Y_{t})\in \epi \crm{t}} \bar{\rho}_{t-1}[Y_{t}]
\end{align*}
repeating recursively the same procedure, proves Proposition~\ref{prop:epi-nested}.
}
\end{proof}